\documentclass[12pt]{article}
\usepackage{amsmath,amssymb,amsfonts,amsthm}
\newenvironment{myabstract}{\par\noindent
{\bf Abstract . } \small }
{\par\vskip8pt minus3pt\rm}
\newcounter{item}[section]
\newcounter{kirshr}
\newcounter{kirsha}
\newcounter{kirshb}
\newenvironment{enumroman}{\setcounter{kirshr}{1}
\begin{list}{(\roman{kirshr})}{\usecounter{kirshr}} }{\end{list}}
\newenvironment{enumarab}{\setcounter{kirshb}{1}
\begin{list}{(\arabic{kirshb})}{\usecounter{kirshb}} }{\end{list}}
\newtheorem{theorem}{Theorem}[section]

\newtheorem{lemma}[theorem]{Lemma}
\newtheorem{corollary}[theorem]{Corollary}
\newenvironment{demo}[1]{\noindent{\bf #1.}\upshape\mdseries}
{\nopagebreak{\hfill\rule{2mm}{2mm}\nopagebreak}\par\normalfont}
\theoremstyle{definition}

\newtheorem{example}[theorem]{Example}
\newtheorem{definition}[theorem]{Definition}

\def\C{{\mathfrak{C}}}

\def\At{{\bf At}}
\def\Nr{{\mathfrak{Nr}}}

\def\Sg{{\mathfrak{Sg}}}

\def\A{{\mathfrak{A}}}
\def\B{{\mathfrak{B}}}
\def\C{{\mathfrak{C}}}
\def\D{{\mathfrak{D}}}
\def\M{{\mathfrak{M}}}
\def\N{{\mathfrak{N}}}

\def\CA{{\bf CA}}

\def\SC{{\bf SC}}
\def\QEA{{\bf QEA}}

\def\PA{{\bf PA}}
\def\PEA{{\bf PEA}}

\def\Rd{{\ Rd}}
\def\(R)RA{{\bf (R)RA}}
\def\RA{{\bf RA}}

\def\Sc{{\bf Sc}}

 \def\CA{{\sf CA}}
\def\B{{\sf B}}
\def\G{{\sf G}}

 \def\Cm{{\mathfrak{Cm}}}
\def\Nr{{\mathfrak{Nr}}}
\def\SNr{{\bf S}{\mathfrak{Nr}}}

\def\Ra{{\mathfrak{Ra}}}

\def\Ra{{\mathfrak{Ra}}}
\def\Nr{{\mathfrak{Nr}}}

\def\A{{\mathfrak{A}}}
\def\B{{\mathfrak{B}}}
\def\C{{\mathfrak{C}}}
\def\D{{\mathfrak{D}}}

\def\A{{\mathfrak{A}}}
\def\B{{\mathfrak{B}}}
\def\C{{\mathfrak{C}}}
\def\D{{\mathfrak{D}}}

\def\GG{{\mathfrak{GG}}}

\def\L{{\mathfrak{L}}}
\def\Rd{{\mathfrak{Rd}}}

\def\At{{\mathfrak{At}}}
\def\L{{\mathfrak{L}}}
\def\Bl{{\mathfrak{Bl}}}
\def\CA{{\bf CA}}
\def\RA{{\bf RA}}

\def\G{{\bf G}}

\def\F{{\mathfrak{F}}}
\def\At{{\sf{At}}}
\def\N{\mathbb{N}}

\def\Ra{{\mathfrak{Ra}}}
\def\Nr{{\mathfrak{Nr}}}

\def\CA{{\bf CA}}

\def\Cm{{\mathfrak Cm}}
\def\Sg{{\mathfrak Sg}}
\def\P{{\mathfrak P}}

\def\N{{\cal N}}

\def\At{{\sf At}}

\title{Some results on polyadic algebras}
\author{Tarek Sayed Ahmed \\
Department of Mathematics, Faculty of Science,\\ 
Cairo University, Giza, Egypt.
  }
\begin{document}
\maketitle

\begin{myabstract} While every polyadic algebra ($\PA$) of dimension 
$2$ is representable, we show that not every atomic polyadic algebra of dimension two is completely representable; though the class is elementary.
Using  higly involved constructions of Hirsch and Hodkinson we show that it is not elementary for higher dimensions
a result that, to the best of our knowledge, though easily destilled from the literature, was never published.
We give a uniform flexible way of constructing weak atom structures that are not strong, and we discuss the possibility of extending such result to infinite 
dimensions. Finally we 
show that for any finite $n>1$, there are two $n$ dimensional polyadic atom structures $\At_1$ and $\At_2$ 
that are $L_{\infty,\omega}$ equivalent, and there exist
atomic $\A,\B\in \PA_n$, such that  $\At\A=\At_1$ and $\At\B= \At_2$, $\A\in \Nr_n\PA_{\omega}$ and $\B\notin \Nr_n\PA_{n+1}$.
This can also be done for infinite dimensions (but we omit the proof). 
\end{myabstract}

\section{Introduction}

There are two main algebraisations of first order, cylindric algebras due to Tarski and polyadic algebras due to Halmos. 
In the infinite dimensional, case they are significantly distinct, and it commonly accepted that they actually 
belong to two different universes
or paradigms. One blatant difference is that polyadic algebras have continuum many opeartions, whie cylindric algebras have only countably many.

It is hard to give a rigouous mathematical definition of such a dichotomy, but they more often than not manifest contradictory 
behaviour. A plathora of  results, existing in the literature, see \cite{Sagi}, point out to the fact that there is some kind of dichotomy 
which does not need further rigorous evidence.

For finite dimensions, they are pretty close, 
but there are differences that are delicate and quite subtle. 
In the context of presence of equality for example, as is the case with cylindric algebras,
when we are dealing with polyadic algebras endowed with diagonal elements, 
the substitution operations corresponding to transpositions,
are not even finitely axiomatizable over their diagonal free case, this happens for all dimensions $>2$,
which suggest beyond doubt,
that the substitution operations, a basic operation in polyadic algebras, that are not 
definable in cylindric algebras, adds a lot, and consequently
one can expect that there are also non-trivial differences between the finite dimensional cylindric  algebras and polyadic ones 
(Polyadic algebras occur in the literature under the names of quasi polyadic or finitary polyadic algebras; all three are essentially the same.)

The standard way of obtaining cylindric algebras, known as set algebras, is from a model for first order logic wih equality.
Polyadic algebras, are rather obtained from models of first order logic {\it without} equality, and from, at least such a perspective,
they are different, in so much as first order logic without is different that without equality.

On the other hand, because for some reason or another , maybe historic or aesthetic, cylindric algebras
got the bigger share of research,  during the lase decated, 
yielding plenty sophisticated deep results, using graph theory and finite combinatorics. 
And there is this feeling in the air that many results obatined for cylindric algebras, carry over to polyadic algebras without much ado, 
but  no one has actually bothered to make sure that this is the case indeed, possibly under the conviction that is a systematic boring task. 
In some case, it definitely is, but in other, as it turns out this is not true at all, like for example 
the polyadic algebras constructed by Sayed Ahmed and Robin hirsch to confirm the analogue of the famous neat embeding pproblem,
already proved for cylindric algebras, to polyadic algebras and other reducts thereof. 
Another is that, contrary to cylindric algebras, it is not known where there is a universal axiomatization of the class of representable polyadic algebras, 
of dimension $>2.$ 
Not only that, but such feelings often based on false intuition can be downright wrong, as we show in a minute.

\section{A false impression about polyadic algebras}

We give an example of cylindrfier free reducts of polyadic algebras. 
This is taken from [AGNS]. However, in the latter reference the example worked out by Andreka et all,
addresses polyadic algebras, and it works only for dimension two. Here it works for all dimensions.
Besides it answers a question of Hodkinson's, for Pinter's algebras.

\begin{theorem} For any ordinal $\alpha>2$, and any infinite cardinal $\kappa$, 
there is an atomic algebra $\A\in SA_{\alpha}$, with $|A|=\kappa$,  
that is not  completely representable.In particular, $\A$ can be countable.
\end{theorem}

\begin{proof}

It suffices to show that there is an algebra $\A$, and a set $S\subseteq A$, such that $s_0^1$ does not preserves $\sum S$.
For if $\A$ had a representation as stated in the theorem, this would mean that $s_0^1$ is completely additive in $\A$.

For the latter statement, it clearly suffices to show that if $X\subseteq A$, and $\sum X=1$,
and there exists  an injection $f:\A\to \wp(V)$, such that $\bigcup_{x\in X}f(x)=V$,
then for any $\tau\in {}^nn$, we have $\sum s_{\tau}X=1$. So fix $\tau \in V$ and assume that this does not happen.
Then there is a $y\in \A$, $y<1$, and
$s_{\tau}x\leq y$ for all $x\in X$.
(Notice that  we are not imposing any conditions on cardinality of $\A$ in this part of the proof).
Now
$$1=s_{\tau}(\bigcup_{x\in X} f(x))=\bigcup_{x\in X} s_{\tau}f(x)=\bigcup_{x\in X} f(s_{\tau}x).$$
(Here we are using that $s_{\tau}$ distributes over union.)
Let $z\in X$, then $s_{\tau}z\leq y<1$, and so $f(s_{\tau}z)\leq f(y)<1$, since $f$ is injective, it cannot be the case that $f(y)=1$.
Hence, we have
$$1=\bigcup_{x\in X} f(s_{\tau}x)\leq f(y) <1$$
which is a contradiction, and we are done.
Now we turn to constructing the required  counterexample, which is an easy adaptation of a 
construction due to Andr\'eka et all in [AGMNS] to our present situation.
We give the detailed construction. One reason is  for the reader's conveniance.

The other, which is more important, is that there are two major differences between our constrcustion and the forementioned one by Andrek et all. 
One is that our constructed algebra can have any infinite cardinality, this is not to much of a change. 
It has to do with enlarging an indexing set of a partition of the large enough base.

The second this, is that
our construction works for {\it all} dimensions, and not just $2$, because we are fortunate enough not to have cylindrifiers.

Now we start implementing our example. 
Let $\alpha$ be the given ordinal.  Let $|U|=\mu$ be an infinite set and $|I|=\kappa$ be a cardinal such 
that $Q_n$, $n\in \kappa$,  is a family of $\alpha$-ary relations that form a partition of $V={}^{\alpha}U^{(p)}$, 
for some fixed sequence $p\in {}^{\alpha}U$. 
Let $i\in I$, and let $J=I\sim \{i\}$. Then of course $|I|=|J|$. Assume that $Q_i=D_{01}=\{s\in V: s_0=s_1\},$
and that each $Q_n$ is symmetric; that is for any $i,j\in n$, $S_{ij}Q_n=Q_n$. 
It is straightforward to show that such partitions exist.


Now fix $F$ a non-principal ultrafilter on $J$, that is $F\subseteq \mathcal{P}(J)$. 
For each $X\subseteq J$, define
\[
 R_X =
  \begin{cases}
   \bigcup \{Q_k: k\in X\} & \text { if }X\notin F, \\
   \bigcup \{Q_k: k\in X\cup \{i\}\}      &  \text { if } X\in F
  \end{cases}
\]

Let $$\A=\{R_X: X\subseteq I\sim \{i\}\}.$$
Notice that $|\A|\geq \kappa$. Also $\A$ is an atomic set algebra with unit $R_{J}$, and its atoms are $R_{\{k\}}=Q_k$ for $k\in J$.
(Since $F$ is non-principal, so $\{k\}\notin F$ for every $k$).
We check that $\A$ is indeed closed under the operations.
Let $X, Y$ be subsets of $J$. If either $X$ or $Y$ is in $F$, then so is $X\cup Y$, because $F$ is a filter.
Hence
$$R_X\cup R_Y=\bigcup\{Q_k: k\in X\}\cup\bigcup \{Q_k: k\in Y\}\cup Q_0=R_{X\cup Y}$$
If neither $X$ nor $Y$ is in $F$, then $X\cup Y$ is not in $F$, because $F$ is an ultrafilter.
$$R_X\cup R_Y=\bigcup\{Q_k: k\in X\}\cup\bigcup \{Q_k: k\in Y\}=R_{X\cup Y}$$
Thus $A$ is closed under finite unions. Now suppose that $X$ is the complement of $Y$ in $J$.
Since $F$ is an ultrafilter exactly one of them, say $X$ is in $F$.
Hence,
$$\sim R_X=\sim{}\bigcup \{Q_k: k\in X\cup \{0\}\}=\bigcup\{Q_k: k\in Y\}=R_Y$$
so that  $\A$ is closed under complementation (w.r.t $R_{J}$).
We check substitutions. Transpositions are clear, so we check only replacements. It is not too hard to show that
\[
 S_0^1(R_X)=
  \begin{cases}
   \emptyset & \text { if }X\notin F, \\
   R_{\mathbb{Z}^+}      &  \text { if } X\in F
  \end{cases}
\]

Now
$$\sum \{S_0^1(R_{k}): k\in J\}=\emptyset.$$
and
$$S_0^1(R_{J})=R_{J}$$
$$\sum \{R_{\{k\}}: k\in J\}=R_{J}=\bigcup \{Q_k:k\in J\}.$$
Thus $$S_0^1(\sum\{R_{\{k\}}: k\in J\})\neq \sum \{S_0^1(R_{\{k\}}): k\in J\}.$$
The algebra required is that generated by the $\kappa$ many atoms.
Finally, this algebra cannot posses a complete represenation, for any such representation implies the complete 
additivity of the substitution operations as indicated above.

\end{proof}
The following answers a question of Hodkinson's.
\begin{corollary} By discarding replacements, we obtain that Pinter's atomic algebras may not be completely representable
\end{corollary}
There is a wide spread belief, almost permenantly established that like cylindric algebras, any atomic {\it poyadic algebras of dimension $2$}
is completely representable. This is wrong.  The above example, indeed shows that it is not the case, 
because the set algebras consrtucted above , if we impose the additional condition that each $Q_n$ has $U$ as its domain and range,
 then the algebra in question becomes closed under the first two cylindrfiers, and by the same reasoning as above,
it {\it cannot} be completely representable.
\begin{theorem} The class of atomic polyadic algebras of dimension $2$ is elementary. 
schema
\end{theorem}
\begin{proof} Let $\At(x)$ be the first order formula asserting that $x$ is an atom, namely, $\At(x)$ is the formula
$x\neq 0\land (\forall y)(y\leq x\to y=0\lor y=x)$. We first  assume that $n>1$ is finite, the other cases degenerate to the Boolean case.
For distinct $i,j<2$ let $\psi_{i,j}$ be the formula:
$y\neq 0\to \exists x(\At(x)\land s_i^jx\neq 0\land s_i^jx\leq y).$ Let $\Sigma$ be obtained from the axiomatization
$PA_n$ by adding $\psi_{i,j}$
for every distinct $i,j\in 2$.
Then $CRPA_2={\bf Mod}(\Sigma)$. 
\end{proof}

\section{Weakly representable atom structures that are not strongly representable}

For a fixed  graph $\G$, we define a family of labelled graphs $\cal F$ such that every edge of each graph $\Gamma\in {\cal F}$, 
is labelled  by a unque label from
$\G\cup \{\rho\}$, $\rho\notin \G$. Then one forms a labelled graph $M$ which can be viewed as model of a natural signiture, 
namely, the one with relation symbols $(a, i)$, for each $a \in \G \cup \{ \rho \}$ and
$ i < n$.  This $M$ can be constructed as a limit of finite structures, in the spirit of Fraisse constructions.
Then one takes a subset $W\subseteq {}^nM$, by roughly dropping assignments that do not satify $(\rho, l)$ for every $l<n$.
Formally, $W = \{ \bar{a} \in {}^n M : M \models ( \bigwedge_{i < j < n,
l < n} \neg (\rho, l)(x_i, x_j))(\bar{a}) \}.$
All this can be done with an arbirary graph.

Now for particular choices of $\G$; the algebra relativized set algebras based on $M$, but taking 
only sequences in $W$ in $L_n$ is an atomic 
representable algebra. This algebra has universe $\{\phi^M: \phi\in L_n\}$ where $\phi^M=\{s\in W: M\models \phi[s]\}.$
Its completion is the relatvized sets algebras consisting of $\phi^M$, $\phi\in L_{\infty,\infty}$, 
which turns out not representable. (All logics are taken in the above signature).
In fact, we will show that for certian choices of $\G$, it will not be even in $\SNr_nCA_{n+2}$.
Let us get more technical.

\begin{example}

\begin{enumarab}

\item A \textit{labelled graph} is an undirected graph $\Gamma$ such that
every edge ( \textit{unordered} pair of distinct nodes ) of $\Gamma$
is labelled by a unique label from $(\G \cup \{\rho\}) \times n$, where
$\rho \notin \G$ is a new element. The colour of $(\rho, i)$ is
defined to be $i$. The \textit{colour} of $(a, i)$ for $a \in \G$  is $i$.
Now we define a class $\GG$ of certain labelled graphs.The class $\GG$ consists of all complete labelled graphs $\Gamma$ (possibly
the empty graph) such that for all distinct $ x, y, z \in \Gamma$,
writing $ (a, i) = \Gamma (y, x)$, $ (b, j) = \Gamma (y, z)$, $ (c,
l) = \Gamma (x, z)$, we have:\\
\begin{enumarab}
\item $| \{ i, j, l \} > 1 $, or
\item $ a, b, c \in \G$ and $ \{ a, b, c \} $ has at least one edge
of $\G$, or
\item exactly one of $a, b, c$ -- say, $a$ -- is $\rho$, and $bc$ is
an edge of $\G$, or
\item two or more of $a, b, c$ are $\rho$.
\end{enumarab}

\item There is a countable labelled graph $M\in \GG$ with the following
property:\\

If $\triangle \subseteq \triangle' \in \GG$, $|\triangle'|
\leq n$, and $\theta : \triangle \rightarrow M$ is an embedding,
then $\theta$ extends to an embedding $\theta' : \triangle'
\rightarrow M$. 

Let $L^+$ be the signature consisting of the binary
relation symbols $(a, i)$, for each $a \in \G \cup \{ \rho \}$ and
$ i < n$. Let $L = L^+ \setminus \{ (\rho, i) : i < n \}$. From now
on, the logics $L^n, L^n_{\infty \omega}$ are taken in this
signature.
We may regard any non-empty labelled graph equally as an
$L^+$-structure, in the obvious way.

\item Let $W = \{ \bar{a} \in {}^n M : M \models ( \bigwedge_{i < j < n,
l < n} \neg (\rho, l)(x_i, x_j))(\bar{a}) \}.$
For an $L^n_{\infty \omega}$-formula $\varphi $, we define
$\varphi^W$ to be the set $\{ \bar{a} \in W : M \models_W \varphi
(\bar{a}) \}$, an we let $\A$ to be the relativised set algebra with domain
$$\{\varphi^W : \varphi \,\ \textrm {a first-order} \;\ L^n-
\textrm{formula} \}$$  and unit $W$, endowed with the algebraic
operations ${\sf d}_{ij}, {\sf c}_i, $ ect., in the standard way .
 Fix finite $N\geq n(n-1)/2$. 

$\G$ can be any graph that contains infinitely countably many cliques (complete subgraphs) each of size $N$.
For example it can be $\G=(\N,E)$ with nodes $\N$ and $i,l$ is an edge i.e $(i,l)\in E$ if  
$0<|i-l|<N$, or a countable union of cliques, denote by $N\times \omega$. 

Now let $G$ be an infinite countable graph that contains infinitely many $N$ cliques. 
Then $\A$ is a representable (countable) atomic polyadic algebra but $\Rd_{ca}{\cal C}\notin S\Nr_n\CA_{n+2}$, its complex lagebra is isomorphic
to the algebra consisting of formula in $L_{\infty}$ is not representable. Further, $\A$ is acually isomorphic to the term algebra over 
its atom structure. 

\end{enumarab}
\end{example}
The above example can be easily transferred to relation algebras as follows:
\begin{example}

We define a relation algebra atom structure $\alpha(\G)$ of the form
$(\{1'\}\cup (\G\times n), R_{1'}, \breve{R}, R_;)$.
The only identity atom is $1'$. All atoms are self converse, 
so $\breve{R}=\{(a, a): a \text { an atom }\}.$
The colour of an atom $(a,i)\in \G\times n$ is $i$. The identity $1'$ has no colour. A triple $(a,b,c)$ 
of atoms in $\alpha(\G)$ is consistent if
$R;(a,b,c)$ holds. Then the consistent triples are $(a,b,c)$ where

\begin{itemize}

\item one of $a,b,c$ is $1'$ and the other two are equal, or

\item none of $a,b,c$ is $1'$ and they do not all have the same colour, or

\item $a=(a', i), b=(b', i)$ and $c=(c', i)$ for some $i<n$ and 
$a',b',c'\in \G$, and there exists at least one graph edge
of $G$ in $\{a', b', c'\}$.

\end{itemize}
$\alpha(\G)$ can be checked to be a relation atom structure. 
The atom structure of $\Rd_{ca}\A$ is isomorphic (as a cylindric algebra
atom structure) to the atom structure ${\cal M}_n$ of all $n$-dimensional basic
matrices over the relation algebra atom structure $\alpha(\G)$.
Indeed, for each  $m  \in {\cal M}_n, \,\ \textrm{let} \,\ \alpha_m
= \bigwedge_{i,j<n}  \alpha_{ij}. $ Here $ \alpha_{ij}$ is $x_i =
x_j$ if $ m_{ij} = 1$' and $R(x_i, x_j)$ otherwise, where $R =
m_{ij} \in L$. Then the map $(m \mapsto
\alpha^W_m)_{m \in {\cal M}_n}$ is a well - defined isomorphism of
$n$-dimensional cylindric algebra atom structures.
We can show that thae $\Cm\alpha(\G)$ is not representable like exactly [weak]
using Ramseys theore.
Here we show something stronger.

\end{example}

\begin{theorem}
We have $\Cm\alpha(\G)$ is not in $S\Ra \CA_{n+2}$. 
\end{theorem}
\begin{proof} The idea is to use relativized representations. Such algebras are 
localy representable, but the epresentation is global enough so that Ramseys theorem applies. 
Hence the full complex cylindric algebra over the set of $n$ by $n$ basic matrices
- which is isomorphic to $\cal C$ is not in $S\Nr_n\CA_{n+2}$ for we have a relation algebra
embedding of $\Cm\alpha(\G)$ onto $\Ra\Cm{\cal M}_n$.
Assume for contradiction that $\Cm\alpha(\G)\in S\Ra \CA_{n+2}$. Then $\Cm\alpha(\G)$ has an $n$-flat representation $M$ \cite{HH}  13.46, 
which is $n$ square \cite{HH} 13.10. 
In particular, there is a set $M$, $V\subseteq M\times M$ and $h: \Cm\alpha(\G)\to \wp(V)$ 
such that $h(a)$ ($a\in \Cm\alpha(\G)$) is a binary relation on $M$, and
$h$ respects the relation algebra operations. Here $V=\{(x,y)\in M\times M: M\models 1(x,y)\}$, where $1$ is the greatest element of 
$\Cm\alpha(\G)$.
A clique $C$ of $M$ is a subset of the domain $M$ such that for $x,y\in C$ we have $M\models 1(x,y)$, equivalently $(x,y)\in V$.
Since $M$ is $n+2$ square, then for all cliques $C$ of $M$ with $|C|<n+2$, all $x,y\in C$ and $a,b\in \Cm\alpha(\G)$, $M\models (a;b)(x,y)$ 
there exists $z\in M$ such that $C\cup \{z\}$ is a clique and $M\models a(x,z)\land b(z,y)$.
For $Y\subseteq \N$ and $s<n$, set 
$$[Y,s]=\{(l,s): l\in Y\}.$$
For $r\in \{0, \ldots N-1\},$ $N\N+r$ denotes the set $\{Nq+r: q\in \N\}.$
Let $$J=\{1', [N\N+r, s]: r<N,  s<n\}.$$
Then $\sum J=1$ in $\Cm\alpha(\G).$
As $J$ is finite, we have for any $x,y\in M$ there is a $P\in J$ with
$(x,y)\in h(P)$.
Since $\Cm\alpha(\G)$ is infinite then $M$ is infinite. 
By Ramsey's Theorem, there are distinct
$x_i\in X$ $(i<\omega)$, $J\subseteq \omega\times \omega$ infinite
and $P\in J$ such that $(x_i, x_j)\in h(P)$ for $(i, j)\in J$, $i\neq j$. Then $P\neq 1'$. 
Also $(P;P)\cdot P\neq 0$. 
This follows from $n+2$ squareness and that if $x,y, z\in M$, 
$a,b,c\in \Cm\alpha(\G)$, $(x,y)\in h(a)$, $(y, z)\in h(b)$, and 
$(x, z)\in h(c)$, then $(a;b)\cdot c\neq 0$. 
A non -zero element $a$ of $\Cm\alpha(\G)$ is monochromatic, if $a\leq 1'$,
or $a\leq [\N,s]$ for some $s<n$. 
Now  $P$ is monochromatic, it follows from the definition of $\alpha$ that
$(P;P)\cdot P=0$. This contradiction shows that 
$\Cm\alpha(\G)$ is not in $S\Ra\CA_{n+2}$. Hence $\Cm{\cal M}_n\notin S\Nr_n\CA_{n+2}$.

\end{proof}

We have not seen a publication of ths result, though its proof can be easily destilled from known rather involved proofs.

\begin{lemma} Let $\D$ be a polyadic equality algebra of dimension $n\geq 3$, that is generated by the set $\{x\in D: \Delta x\neq n\}.$
Then if $\Rd_{qa}\D$ is completely representable, then so is $\D$.
\end{lemma}
\begin{proof} First suppose that $\D$ is simple, and let $h: \D\to \wp(V)$ be a complete representation, where $V=\prod_{i<n}U_i$ for sets $U_i$.
We can assume that $U_i=U_j$ for all $i,j<n$, and if $s\in V$, $i,j<n$ and $a_i=a_j$ then $a\in h(d_{ij}$. Indeed, let $\delta =\prod d_{ij}\in \D$. 
As $C$ is a cyilndric algebra, we have $c_{(n)}\delta =1$, so for each $u\in U_i$ there is an $s\in h(\delta)$ with $a_i=u$. 
So there exists a function $s_i: U_i\to h(\delta)$ such that $(s_i(u))_i=u$ for each $u\in U_i$.

Let $U$ be the disjoint union of the $U_i$s. Let $t_i:U\to U_i$ be the surjection defined by $t_i(u)=(s_j(u))_i$.
Let $g: \D\to \wp(^nU)$ be defined via 
$$d\mapsto \{s\in {}^nU: (t_0(a_0),\ldots, t_{n-1}(a_{n-1}))\in h(d)\}.$$
Then $g$ is a complete representation of $\D.$ Now suppose $s\in {}^nU$, satisfies $s_i=s_j$ with $a_i\in U_k$, say, where $k<n$. 
Let $\bar{b}=s_k(a_i)=s_k(a_j)\in h(\delta).$ Then $t_i(a_i)=b_i$ and $t_j(a_j)=b_j$, so $(t_i(a_i): i<n)$ 
agrees with $\bar{b}$ on coordinates $i,j$. Since $\bar{b}\in h(\delta)$ and 
$\Delta d_{ij}=\{i,j\}$, then $(t_i(a_i): i<n)\in h(d_{ij}$ and so $s\in g(d_{ij}),$ as required.

Now define $\sim_{ij}=\{(a_i, a_j): \bar{a}\in h(d_{ij})$. Then it easy to 
check that $\sim_{01}=\sim_{i,j}$ is an equivalence relation on $U$. For 
$s,t\in {}^nU$, define $s\sim t$, if $s_i\sim t_i$ for each $i<n$, then $\sim$ is an equivalence relation
on $^nU$. Let
$$E=\{d\in D: h(d)\text { is a union of $\sim$ classes }\}.$$
Then $$\{d\in D: \Delta d\neq n\}\subseteq E.$$

Furthermore, $E$ is the domain of a complete subalgebra of $\C$. 
Let us check this. We have $\{0,1, d_{ij}: i,j<n\}\subseteq E$, since $\Delta 0=\Delta 1=\emptyset$ 
and $\Delta d_{ij}=\{i,j\}\neq n$ (as $n\geq 3$).
If $h(d)$ is a union of $\sim$ classes, then so
is $^nU\setminus h(d)= h(-d)$. If $S\subseteq E$ and $\sum S$ exists in $\D$, then because $h$ is complete representation 
we have $h(\sum^{\D}S)=\bigcup h[S]$, a union of $\sim$ classes
so $\sum S\in E$. Hence $E=C$. Now define $V=U/\sim_{01}$, and
define $g:\C\to \wp(^nV)$ via
$$c\mapsto \{(\bar{a}/\sim_{01}): \bar{a}\in h(c).$$
Then $g$ is a complete representation.

Now we drop the assumption that $\D$ is simple. Suppose that $h:\D\to \prod_{k\in K} Q_k$ is a complete representation. 
Fix $k\in K$, let $\pi_k: Q\to Q_k$ be the canonical projection, and let 
$\D_k=rng(\pi_k\circ h)$. We define diagonal elements in $\D_k$ by $d_{ij}=\pi_k(h^{\C}(d_{ij}))$. This expands $\D_k$ 
to a cylindric-type algebra $\C_k$ that is a homomorphic image of $\C$, and hence
is a cylindric algebra with diagonal free reduct $\D_k$. Then the inclusion map $i_k:\D_k\to Q_k$ is a complete
representation of $D_k$. 
Since obviously
$$\pi_k[h[\{c\in C: \Delta c\neq n\}]\subseteq \{c\in C_k: \Delta c\neq n\}$$ 
and $\pi_k, h$ preserve arbitrary sums, then $C_k$ is completely generated by $\{c\in C_k: \Delta c\neq n\}$. 
Now $c_{(n)}x$ is a discriminator term in $Q_k$, so $D_k$ is simple.
So by the above $\C_k$ has complete represenation $g_k:\C_k\to Q_k'$. Define
$g: \C\to \prod_{k\in K}Q_k'$ via
$$g(c)_k=g_k(\pi_k(h(c))).$$
Then $g$ defines a complete representation.
\end{proof}


\begin{example}

In definition 3.6.3 \cite{HH} a cylindric atom structure is defined from a family $K$ of $L$ structures, closed under forming subalgebra. This class is 
formulated in a language
$L$ of relation symbols $<n$. Call this atom structure $\rho(K)$. 
The atom structure,  can be turned easily into a polyadic equality atom structure 
by defining accesibility relations correponding to the substituton $s_{i,j}$ by: $R_{ij}=\{[f], [g]: f, g \in {\cal F}: f=g\circ [i,j]\}.$

Two examples are given of such clases, what concerns us is the second (rainbow) class defined in 3.6. 9, referred to as classes
based on on graph.
Fix a graph $\Gamma$. The rainbow polyadic equality algebra based on this graph is denoted by $R(\Gamma)$ 
is the complex algebra of $\rho(K(\Gamma))$, namely $\Cm\rho(K(\Gamma))$.
It is proved that If $\Gamma$ is a countable graph, then the cylindric algebr $R(\Gamma)$ is completely reprsentable
if and only if $\Gamma$ contains a reflexive node or an infinite clique, 
This proof can be checked to work for polyadic equality algebras, and by our previous lemma, it also works for
polyadic algebras.

Define $K_k$ and $\Gamma$ as in corollary 3.7.1 in \cite{HH}. Then $R(\Gamma)$ is s completely representable.
But $\Gamma$ has arbirary large cliques, hence it is 
elementay equivalent to a countable graph $\Delta$ with an infinite clique. Then $R(\Delta)\equiv R(\Gamma)$,
and by the above chracteization the latter is completely representable, the former is not.
Notice that $\Delta\equiv \Gamma$ as first order structures.
\end{example}


\subsection{The infinite dimensional case}

Let us try to extend the result concerning existence of weaky representable atom structures that are not strongly so.
If we insist on using graphs and model theory we wil have to change our base logic, to allows
infnitary formulas.
For simplicity we consider the arity of formulas to be at most
$\omega$. $L^{\omega}$ is a quantifier logic that allows infinitary predicates of arbitrary rank, 
and otherwise is like first order logic, in particular quantification can be taken only on finitely many 
variables. $L^{\omega}_{\infty}$ is the logic obtained from $L^{\omega}$
by adding infinite conjunctions. Let $\G$ be a graph.

\begin{enumarab}

\item A \textit{labelled graph} is an undirected graph $\Gamma$ such that
every edge ( \textit{unordered} pair of distinct nodes ) of $\Gamma$
is labelled by a unique label from $(\G \cup \{\rho\}) \times \omega$, where
$\rho \notin \G$ is a new element. The colour of $(\rho, i)$ is
defined to be $i$. The \textit{colour} of $(a, i)$ for $a \in \G$  is $i$.
Now we define a class $\GG$ of certain labelled graphs.The class $\GG$ consists of all complete labelled graphs $\Gamma$ (possibly
the empty graph) such that for all distinct $ x, y, z \in \Gamma$,
writing $ (a, i) = \Gamma (y, x)$, $ (b, j) = \Gamma (y, z)$, $ (c,
l) = \Gamma (x, z)$, we have:\\
\begin{enumarab}
\item $| \{ i, j, l \} > 1 $, or
\item $ a, b, c \in \G$ and $ \{ a, b, c \} $ has at least one edge
of $\G$, or
\item exactly one of $a, b, c$ -- say, $a$ -- is $\rho$, and $bc$ is
an edge of $\G$, or
\item two or more of $a, b, c$ are $\rho$.
\end{enumarab}

\item There is a countable labelled graph $M\in \GG$ with the following
property:\\

If $\triangle \subseteq \triangle' \in \GG$, $|\triangle'|
\leq n$, and $\theta : \triangle \rightarrow M$ is an embedding,
then $\theta$ extends to an embedding $\theta' : \triangle'
\rightarrow M$. 

Let $L^+$ be the signature consisting of the binary
relation symbols $(a, i)$, for each $a \in \G \cup \{ \rho \}$ and
$ i < \omega$. Let $L = L^+ \setminus \{ (\rho, i) : i < \omega \}$. From now
on, the logics $L^{\omega}, L^{\omega}_{\infty}$ are taken in this
signature.
Fix $p\in {}^{\omega}M$,  and let $V={}^{\omega}M^{(p)}$. For $\bar{a}\in V$ and $\phi\in L^{\omega}_{\infty}$ 
satifiability is defined the usual Tarskian 
way. For a formula $\phi$, we write $\phi^{\M}$ for all asignments in $V$ that satisfy $M$.

\item Let $W = \{ \bar{a} \in {}V : M \models ( \bigwedge_{i < j < \omega,
l <\omega} \neg (\rho, l)(x_i, x_j))(\bar{a}) \}.$
For an $L^{\omega}_{\infty}$ formula $\varphi $, we define
$\varphi^W$ to be the set $\{ \bar{a} \in W : M \models_W \varphi
(\bar{a}) \}$, an we let $\A$ to be the relativised set algebra with domain
$$\{\varphi^W : \varphi \,\ \textrm {an} \;\ L^{\omega}-
\textrm{formula} \}$$  and unit $W$, endowed with the algebraic
operations ${\sf d}_{ij}, {\sf c}_i, $ ect., in the standard way .
Let $\C$ be the algebra with base $\varphi^W$ in $L^{\omega}_{\infty}$ and operations as above. 

For $\A$ to a a representable (countable) atomic polyadic algebra, we need a graph with arbitrary large cliques. For 
$\C$ its completion to be non-representable, we need a finite chromatic number to apply Ramseys theorem. These two conditions are incompatible.
However, it might be possible in this context, to use the Erdos-Rado theorem, 
extending Ramseys theorem to the uncountable case, by noting that a representation of the complex algebra must have an uncountable 
base.
\end{enumarab}

\section{Two polyadic atom structures equivalent in $L_{\infty, \omega}$, generating (in their complex algebra)
two polyadic algebras one in $\Nr_n\PA_{\omega}$
and the other not in $\Nr_n\PA_{n+1}$.}

A class closely related to the class of completely representable algebras is that of neat reducts; 
the completely representable algebras are those tha have a strong neat embedding 
property. This characterization works even for the infinite dimensional case, if we take weak structures. But in all cases it only adresses 
the countable case \cite{OTT}, \cite{Sayed}.

Both classes are non elementary for all dimensions $>2$.
For quasipolyadic algebras of infinite dimensions, however,  it is not known whether atomic algebras are completely representable or not.
This is anther result for which there is an unbased feeling in the air that it is true.
Both classes are psuedo-elementary.

But now we show that there is a very important diference.
An atom structure which is completely representable, have all atomic algebras based on it completely representable, 
but this is not the case for neat reducts. The former class is not elementary, and it seems that the class of atom structures 
for which algebras based are neat reducts is also not- elementary. (We are a little bit careless about the number of extra dimensions in the neat reduct,
but its variation leads to the richness of the problem. 
We could require that both algebras have the same nuber of extra dimensions, but we can also not asume that. We have not pursued this matter
any further). 
Next we give results results concerning neat reducts, for cylindric and polyadic algebras. 
An atom structure of dimension $\alpha$ is (strongly) neat if (every) some algebra based on this atom structure is in $\Nr_{\alpha}\CA_{\alpha+\omega}$.

\begin{theorem} For every ordinal $\alpha>1$, there exists a neat atom structure of dimension $\alpha$, 
that is not strongly neat.
\end{theorem}
\begin{demo}{Proof}
Let $\alpha>1$ and $\F$ is field of characteristic $0$. 
Let 
$$V=\{s\in {}^{\alpha}\F: |\{i\in \alpha: s_i\neq 0\}|<\omega\},$$
Note that $V$is a vector space over the field $\F$. 
We will show that $V$ is a weakly neat atom structure that is not strongly neat.
Indeed $V$ is a concrete atom structure $\{s\}\equiv _i \{t\}$ if 
$s(j)=t(j)$ for all $j\neq i$, and
$\{s\}\equiv_{ij} \{t\}$ if $s\circ [i,j]=t$.

Let $\C$ be the full complex algebra of this atom structure, that is
$${\C}=(\wp(V),
\cup,\cap,\sim, \emptyset , V, {\sf c}_i,{\sf d}_{ij}, {\sf s}_{ij})_{i,j\in \alpha}.$$  
Then clearly $\wp(V)\in \Nr_{\alpha}\CA_{\alpha+\omega}$.
Indeed Let $W={}^{\alpha+\omega}\F^{(0)}$. Then
$\psi: \wp(V)\to \Nr_{\alpha}\wp(W)$ defined via
$$X\mapsto \{s\in W: s\upharpoonright \alpha\in X\}$$
is an isomomorphism from $\wp(V)$ to $\Nr_{\alpha}\wp(W)$.
We shall construct an algebra $\A$ such that $\At\A\cong V$ but $\A\notin \Nr_{\alpha}\CA_{\alpha+1}$.

Let $y$ denote the following $\alpha$-ary relation:
$$y=\{s\in V: s_0+1=\sum_{i>0} s_i\}.$$
Note that the sum on the right hand side is a finite one, since only 
finitely many of the $s_i$'s involved 
are non-zero. 
For each $s\in y$, we let 
$y_s$ be the singleton containing $s$, i.e. $y_s=\{s\}.$ 
Define 
${\A}\in WQEAs_{\alpha}$ 
as follows:
$${\A}=\Sg^{\C}\{y,y_s:s\in y\}.$$
We shall prove that  
$$\Rd_{SC}\A\notin \Nr_{\alpha}SC_{\alpha+1}.$$ 
That is for no $\mathfrak{P}\in SC_{\alpha+1}$, it is the case that $\Sg^{\C}X$ exhausts the set of all $\alpha$ dimensional elements 
of $\mathfrak{P}$. 
So assume, seeking a contradiction,  that $\Rd_{SC}{\A}\in \Nr_{\alpha}SC_{\alpha+1}$. 
Let $X=\{y_s:s\in y\}$. 
Of course every element of $X,$ being a singleton, is an atom.
Next we show that $\A$ is atomic, i.e evey non-zero element contains a minimal non-zero element.  
Towards this end, let $s\in {}^{\alpha}\F^{(\bold 0)}$ be an arbitrary sequence. 
Then 
$$\langle s_0,s_0+1-\sum_{i>1}s_i,s_i\rangle_{i>1}$$ 
and 
$$\langle \sum_{0< i<\alpha}s_i-1,s_i\rangle_{i\geq 1}$$ are elements in y.
Since 
$$\{s\}={\sf c}_1\{\langle s_0,s_0+1-\sum_{i>1}s_i,s_i\rangle_{i>1}\}\cap
{\sf c}_0\{\langle \sum_{0\neq i<\alpha}s_i-1,s_i\rangle_{i\geq 1}\},$$ 
It follows that $\A$ has the same atom structure.

\end{demo}

\section{Stronger Logics}

We now show that logics like $L_{\kappa, \omega}$ and $L_{\infty,\omega}$ cannot characterize the class of neat 
reducts.The second case is of course much stronger.
The first cast can be destilled from the case in \cite{MLQ}, by simple modifications. 
First we let our language have $\kappa^+$ predicate symbols (instead just countably many). 
In this case $\A_u$, as defined in \cite{MLQ} will
have cardinality $\kappa^+$. Then we alter the $uth$ component, and its permuted versions,
by inserting in a Boolean algebra that $L_{\kappa, \omega}$ equivalent to $\A$, whose 
cardinality is $\kappa$. The rest of the proof works. 

But now we prove the stronger result and this needs a more drastic change. 
We will make our components {\it atomic} Boolean algebras, and for this we require that
the basic relations defined in \cite{MLQ}  not only distinct,  but {\it disjoint.}. This is necessary if we want atomic algebras. 
We use a different more basic method to contruct our desired model, which has apperaed in previous publications of ours , in related contexts; and has proved to be quite a nut cracker in these kinds of problems. We include proof for the 
readers conveniance.

$(R,+)$ denotes an arbitray uncountable group, and  
$n=\{0,\cdots,n-1\}$ denotes 
a fixed finite ordinal $ >1.$
\begin{definition}
Let $k<\omega$. Then $S(n,k)$ 
denotes the set of sequences $\langle  i_0,\cdots,i_{n-1}\rangle$ 
such that $i_0\leq i_1\cdots\leq i_{n-1}=k$.
$Cof^+(R)$ denotes the set of all nonempty finite or cofinite subsets of
$R$, i.e. 
$$Cof^+(R)=\{X\subseteq R: X\text{ is non empty, and $X$ or } R-X \text{ is finite}\}.$$
Let $C_r$ be an $n$-ary relation symbol for every $r\in R$.  
For any finite $X\subseteq R$, we define the formulas:
$$\eta (X) =\lor \{C_r(x_0,\cdots,x_{n-1}): r\in X\},\quad\text{and }$$
$$\eta (R-X)=\neg \eta(X)=\land \{\neg C_r(x_0,\cdots,x_{n-1}): r\in X\}.$$
Let $U$ be a set and $E$ an equivalence relation on $U$. Then we write
$xEy$ if $(x,y)\in E$. We write $xE'y$ if $(x,y)\notin E$. 
Suppose that $E$ has distinct
$n$ equivalence classes, or blocks . Then we write $D_E(x_0,x_1\cdots x_{n-1})$ for the 
formula $\bigwedge_{0\leq i<j<n} x_iE' x_j$ asserting that
$x_i,x_j$ are pairwise unrelated according to $E$, for all $i<j<n$. 
That is for all $s\in {}^nU$, $D_E(s_0,\cdots, s_{n-1})$
iff the $s_i$'s belong to distinct blocks.
\end{definition}

\begin{theorem} There are a set $W$, an equivalence relation
$E$ with $n$ blocks on $W$, and $n$-ary relations $C_r\subseteq {}^nW$
for all $r\in R$,  such that conditions (i)-(v) below hold:

\begin{enumroman}

\item $C_r(w_0,w_1\cdots ,w_{n-1})$ implies $D_E(w_0,w_1\cdots ,w_{n-1})$ for
all $r\in R$, and for all $w_0,\cdots,w_{n-1}\in W$.

\item $C_r(w_0,\cdots,w_{n-1})$ implies $C_r(w_{\pi(0)},\cdots,w_{\pi(n-1)})$ for
all $r\in R$, $w_0,\cdots,w_{n-1}\in W$ and permutation $\pi$ of $n$.

\item For all $r\in R$ and for all $w_0, w_1,\cdots w_{n-2}$ in $W$
such that $w_iE'w_j$ whenever $i<j<n-1$, there exists $w_{n-1}\in W$
such that $C_r(w_0,w_1,\cdots w_{n-1})$.

\item For all $k\in \omega$, for all distinct $w, w_0\cdots, w_{k-1}\in W$,
and for any function \par\noindent
${f:S(n,k)\to Cof^+(R)}$,  
there is a $w_k\in W\smallsetminus \{w_0,\cdots, w_{k-1}\}$
such that
$w_kEw$, i.e. $w_k$ is in the same block as $w$,  and
$$\bigwedge \{D(w_{i_0}, w_{i_1},\cdots,  w)
\implies \eta(f(i))[w_{i_0}, w_{i_1},\cdots,  w_{i_{n-1}}]: i\in S(n,k)\}.$$ 
\item The $C_r$'s are pairwise disjoint.
\item $C_{r_1};C_{r_2}=C_{r_1+r_2}$
\end{enumroman}
\end{theorem}

\begin{demo}{Proof}
We shall construct the structure $\langle W,C_r\rangle_{r\in R}$ 
by a routine step by step fashion.
We note that condition (iii) follows from (iv). Often, however, we will only need 
(and refer to) the weaker condition (iii), hence the redundancy in the formulation of Lemma 1. 
Let $I(^k(|R|))$ be the set of all injections, i.e. one to one functions 
from $k$ to $|R|$. Let
$$Q=\cup \{I(^k(|R|))\times ^{S(n,k)}Cof^{+}(R):k<\omega\}.$$   
Then $|Q|=|R|=\mu$, say. Roughly $Q$ stands for the set of all tasks 
that we have to exhaust.
We will construct a set $W$ with cardinality $\mu$.
$Q$ is intended to represent all the instances of condition (iv) as follows: 
An element of $Q$ is of the form $\langle \alpha_0,\cdots,\alpha_{k-1},f\rangle$, 
where $\alpha_0,\cdots,\alpha_{k-1}<\mu$ and $f:S(n,k)\to Cof^+(R)$.
This represents the instance of (iv) where we take 
$k,w_{\alpha_0},\cdots,w_{\alpha_{k-1}},f$ 
as the {\it concrete} values of the quantified items in (iv). 
Let $\rho$ be an enumeration of $Q$ such that: for all $l<\mu$, for all $q\in Q$, 
there exists $j$, with $l<j<\mu$, such that $\rho(j)=q$.
Such a $\rho$ clearly exists. Fix a well ordering $\prec$ of $R$. 
Let $l<\mu$, and suppose that for all $i<l$ we have already defined the
element $w_i$, and the $n$-ary relation $C_r^i\subseteq {}^nW_i$, where
$W_i=\{w_k:k<i\}$, and $C_r^i$ and $W_i$ satisfy all the conditions
with the possible exception of (iv). 
In the $l$'th step we will make the $\rho(l)$'th instance of (iv) true.
Assume that $\rho(l)=\langle \alpha_0,\cdots,\alpha_{k-1},f\rangle$.
Then $\alpha_0,\cdots,\alpha_{k-1}<\mu$ and 
$f:S(n,k)\rightarrow Cof^{+}R$, for some $k$. ($k=0$ is allowed, too). 
Let $w_l$ be an element not in $\cup \{W_i:i<l\}$.
If there exists $i<k$ such that $l\leq \alpha_i$, then for all 
$r\in R$, we define $C_r^l=\cup \{C_r^i:i<l\}$. Else, $l>\alpha_i$ for all $i<k$. In this case, let 
$v_0=w_{\alpha_0},\cdots,\ v_{k-1}=w_{\alpha_{k-1}}$ and $v_k=w_l$.
For all $r\in R$ we define:
\begin{eqnarray*}X_r^l=\{\langle v_{i_0},\cdots,v_{i_{n-1}}\rangle&:& i_0\leq \cdots\leq i_{n-1}=k\,
\\&&\text{and $r$ is the $\prec$-least element of }f(\langle i_0,\cdots,i_{n-1}\rangle )\}\end{eqnarray*}
and
$$C_r^l=\cup \{C_r^i\cup \{\langle s_{\pi (0)},\cdots,s_{\pi (n-1)}\rangle :s\in X_r^l, \pi 
\text{ is a permutation of } n\}\,:\, i<l\}.$$
Finally we set
$$W=\cup \{W_l:l<\mu\}\quad\text{and } C_r=\cup \{C_r^l:l<\mu\}.$$ 
Now we are going to check that the structure $\langle W,C_r\rangle_{r\in R}$ , so defined, 
satisfies conditions (i)-(v). To this end, for any $l<\mu$ and $r\in R$ let
$$Y_r^l=\{\langle s_{\pi(0)},\cdots,s_{\pi(n-1}\rangle:s\in X_r^l\text { and }\pi \text 
{ is a permutation on } n\}$$
and $$D_r^l=\cup \{C_r^j:j<l\}.$$
Then for all $r\in R$ we have $C_r^l=Y_r^l\cup D_r^l$.
Also, the following are not difficult to check:

(1) $w_l\in Rgs$ (the range of $s$) if $s\in Y_r^l$, and $w_l\notin Rgs$ if $s\in D_r^l$.

(2) $C_r^j\subseteq C_r^m$ if $j\leq m<\mu$.

It is easy to show by induction on $l<\mu$ that for all $r\in R$, $C_r^l$ is symmetric,
and if $s\in C_r^l$ then $s$ satisfies $D_E$, i.e. $s(i)$ and $s(j)$ are in distinct blocks
for $0\leq i<j<n.$ 
Thus $C_r$ satrisfies (i) and is symmetric and so $C_r$ satisfies (ii), too.
Now let $r,p\in R$ be distinct. We want to show by induction on $l<\mu$ 
that $C_r^l$ and $C_p^l$ are disjoint.
Now $D_r^l$ and $D_p^l$ are disjoint by the induction hypothesis and by (2).
By (1), it is therefore, enough to show that $Y_r^l$ and $Y_p^l$ are disjoint.
Assume $s\in Y_r^l$, and let 
$\rho(l)=\langle \alpha_0,\cdots,\alpha_{k-1},f\rangle$ and $v=\langle v_0,\cdots,v_k\rangle 
=\langle w_{\alpha_0},\cdots,w_{\alpha_{k-1}}, w_l\rangle$.
Then $v$ is one to one, since the $w_{\alpha_j}$'s are pairwise distinct
and $w_l\neq w_{\alpha_j}$ for all $j<k$, by its very choice.
It follows thus that there are  a unique $i\in S(n,k)$ and 
permutation $\pi$ of $n$ such that $s=\langle z_{\pi(0)},\cdots,z_{\pi(n-1)}\rangle$, 
where $z=\langle v_{i_0},\cdots,v_{i_{n-1}}\rangle\in X_r^l$.
Thus $r$ is the $\prec$- least element of $f(i)$, by $z\in X_r^l$.
Since  $p\neq r$,  we get that $z\notin X_p^l$, and so $z\notin Y_p^l$.
We have shown that $C_r^l$ and $C_p^l$ are disjoint.
By (2) the $C_r$'s are pairwise disjoint, i.e. condition (v) holds.
Finally we check condition(s) (iv) (and (iii)):
Let $k<\omega$, $w_{\alpha_0},\cdots,w_{\alpha_{k-1}}\in W$ be distinct and let
$w\in W$. Let $f:S(n,k)\to Cof^+R$.
Let $l<\mu$ be such that $\rho(l)=\langle \alpha_0,\cdots,\alpha_{k-1},f\rangle$ 
and $l>\alpha_0,\cdots,l>\alpha_{k-1}$.
Such an $l$ exists by the properties of $\rho$.
Then it is not difficult to check that we constructed the $w_l$ so that
it satisfies $\phi$
$$=\bigwedge \{D(w_{\alpha_{i_0}}, w_{\alpha_{i_1}}\cdots, w_l)\implies \eta(f(i))
(w_{\alpha_{i_0}},w_{\alpha_{i_1}},\cdots,w_{\alpha_{i_{n-2}}},x_{i_{n-1}}): 
i\in S(n,k)\}$$ in $\langle W_l,C_r^l\rangle_{r\in R}.$
By $C_r^l={}^nW_l\cap C_r$ we get that $\phi$ is satisfied
in $\langle W,C_r\rangle_{r\in R}$, as well. 
By this the proof of Lemma 1 is complete.
\end{demo}
\noindent

Notice  that by the construction of $W$, $|W|= |R|$. 
In particular, $W$ is also an uncountable set.
We have excluded the empty set from $Cof^{+} R$ in order that (iv) can be satisfied, 
because $\eta(\emptyset)$ is false for any relations $C_r$.
Notice that condition (iv) in Lemma 1, is a ``saturation condition" 
on $W$. It will be used in the proof of fact 3.1  below, to show that the 
structure $\langle W, C_r\rangle_{r\in R}$
admits elimination of quantifiers in a rather strong sense.
The saturation condition (iv) in words. If we have $k$ distinct elements
of $W$ , then for any block, say $W_i$, of $E$,  and for any {\it prescription}, 
there is an element of this block $W_i$ 
satifying this prescription. A prescription is the following:
Given any $n-1$ elements of the pre-selected 
$k$ elements, if these are in distinct blocks
from each other and from $W_i$
then one of $C_r: r\in X$
holds for them, or none of $C_r:$
$r\in X$ hold for them, where
$X$ is a finite subset of $R$.

Let $U=R\times n$. Let $p(u,r)=\{((s_i, u_i): i<n) \in {}^nU: (s_0\ldots s_{n-1})\in C_r\}$ and let 
$1_u=p(u,T)$.

Let $$\A(n)=\Sg^{\C}\{p(u,r): r\in R\}.$$
Let $1_u=E(u,T)$. For $u\in {}V$, let $\A_u$ denote the relativisation of $\A$ to $1_u$
i.e $$\A_u=\{x\in A: x\leq 1_u\}.$$ $\A_u$ is a boolean algebra.
Also  $\A_u$ is uncountable for every $u\in V$
Define a map $f: \Bl\A\to \P=\prod_{u\in {}V}\A_u$, by
$$f(a)=\langle a\cdot 1_u\rangle_{u\in{}V}.$$
Now each $\A_{u}\cong Cof(R)$ and hence is atomic. Also clearly the $\prod A_u$ is also atomic, its atoms are
$(s_i: i< V)$ such that $s_i\neq 0$ for all except some $j$ where $s_j$ is an atom of $A_j$.

Let $u_0,u_1\in S_3$ be distinct and $u_2=u_1\circ u_0$. Let $J=\{u_0, u_1, s_{[i,j]}u_3, i,j< n\}$. 
Take $\B=\prod_{u=u_0, u_1}{A_u}\times B_{s_{[i,j]}u_2}\times_{u\notin J} A_u$ 
where $\B_v$ is the algebra $Cof(\N)$, for $N$ is an elementary subgroup of $R$.
It is easy to show we expand the language of boolean algebras with constants $1_u: u\in V$ and $d_{i,j}$,
The algebra $\A$ becomes first order interpretable with a one dimensional quantifier free interpretation in $\P$,
and under this interpretaion $\B$ becomes a polyadic equality algebra elementary equivalent to 
$\A(n)$ but is not a neat reduct; we denote it by $\B(n)$.

Now we play a game: we devise a game between $\forall$ (male) and $\exists$(female).
We imagine that $\forall$ wants
to prove that $\A(n)$ is different from $\B(n)$ while $\exists$ tries to show that
$\A(n)$ is the same as $\B(n)$. So their conversation has the form of a game.
Player $\forall$ wins if he manages to find a difference between $\A(n)$ and $\B(n)$
before the play is over; otherwise $\exists$ wins. The game is played in $\mu\leq \omega$
steps. At the $i$th step of a play, player $\forall$ takes one of the structures $\A(n)$, $\B(n)$ 
and 
chooses an atom of this structure; then $\exists$ chooses 
an atom of the other structure. So between them they choose an atom $a_i$ of $\A(n)$ 
and an atom 
$b_i$ of $\B(n)$. Apart from the fact that player $\exists$ 
must choose from the other structure from player $\forall$ at each step, 
both players 
have complete freedom to choose as they please; 
in particular, either player can choose an element which was chosen at an earlier step. 
Player $\exists$ is allowed to see and remember all previous moves in the play.
(As the game theorists would say, this is a game of perfect information.) At the end of 
the play  sequences $\bar{a}=(a_i:i<\mu)$ and $\bar{b}=(b_i: i<\mu)$ have been chosen.
The pair $(\bar{a}, \bar{b})$ is known as the play.
We count the play $(\bar{a}, \bar{b})$ as a win for player $\exists$, 
and we say that $\exists$ wins the play, if there is an isomorphism
$f:\Sg^{\A(n)}ran({\bar{a}})\to \Sg^{\B(n)}ran({\bar {b}})$ such that $f\bar{a}=\bar{b}.$
Let us denote this game by $EF_{\mu}(\A(n),\B(n)).$ (It is an instance of an 
Ehrenfeuch-Fraisse game.)
The more $\A(n)$ is like $\B(n)$ , the better chance player $\exists$ 
has of wining these games. 
For example if player $\exists$ knows about an isomorphism $i:\A(n)\to \B(n)$ 
then she can be sure 
of winning every time. All she has to do to follow the rule is:
Choose $i(a)$ whenever player $\forall$ has 
just chosen an element $a$ of $\A(n)$ and $i^{-1}(b)$ 
whenever player $\forall$ has just chosen 
$b$ from $\B(n)$.  
A strategy for a player in a game is a set of rules which tell 
the player exactly how to move, depending on what has happened earlier in the play.
We say that the player uses the strategy $\sigma$ 
in a play if each of his or her moves obeys the rules of $\sigma$.
We say that $\sigma$ is a winning strategy if 
the player wins every play in which he or she uses $\sigma$.
The game generalizes verbatim to atomic 
boolean algebras with operators.

\begin{definition}
Two atomic structures $\A$ and $\B$ are back and forth equivalent 
if $\exists$ has a winning strategy 
for the game
$EFA_{\omega}(\A,\B)$.  
\end{definition}
Let $At\D$ denotes the set of atoms of $\D$.
There is a useful criterion for two structures to be back and forth equivalent.

\begin{definition}
A back and forth system from $\A$ to $\B$ is a set $I$ of pairs
$(\bar{a}, \bar{b})$ of tuples $\bar{a}$ from $At\A$ and $\bar{b}$ from $At\B$, such that
\begin{enumroman}

\item If $(\bar{a}, \bar{b})$ is in $I$, then $\bar{a}$ 
and $\bar{b}$ have the same length and $(\A,\bar{a})$ and $(\B,\bar{b})$
satisfies the same quantifier free formulas.

\item $I$ is not empty.

\item For every pair $(\bar{a},\bar{b})$ in $I$ and every atom $c$ of $\A$ there is an atom  
$d$ of $\B$ such that $(\bar{a}c, b\bar{d})$ is in $I$ and

\item For every pair $(\bar{a},\bar{b})$ in $I$ and every atom $d$ of $\B$ there is an atom  
$c$ of $A$ such that $(\bar{a}c, b\bar{d})$ is in $I.$ 

\end{enumroman}
\end{definition}
Note that by (i) if $\bar{a}$ and $\bar{b}$ is in $I$ 
then there is an isomorphism $f:\Sg(ran\bar{a})\to \Sg(ran\bar{b})$ such that 
$f(\bar{a})=\bar{b}.$

We write $I^*$ for the set of all such functions corresponding to pairs of tuples of atoms 
in $I$.  The above conditions imply the following
for $J=I^*$.
\begin{enumroman}

\item each $f\in J$ is an isomorphism from a finitely generated 
substructure of $\A$ to a finitely generated substructure of $\B.$

\item $J$ is non empty

\item for every $f\in J$ and $c\in At\A$ there is 
$g\supseteq f$ such that $g\in J$ and $c\in dom(g)$

\item for every $f\in J$ and $d\in At\B$ there is 
$g\supseteq f$ such that $g\in J$ and $d\in ran(g)$

\end{enumroman}
And conversely, it is not hard to see, that  if $J$ 
is any set satisfying then there is a back and forth system $I$ such that $J=I^*.$
The following Theorem is intuitive.

\begin{theorem}  $\A$ and $\B$ are back-and forth equivalent 
if and only if there is a back and forth system from $\A$
to $\B$.
\end{theorem}
\begin{demo}{Proof}Suppose that $\A$ is back and forth equivalent to $\B$, so 
that player $\exists$ has a winning strategy $\sigma$
for the game $EF_{\omega}(\A,\B)$. Then define
$I$ to consist of all pairs of tuples of atoms which are of the form 
$(\bar{c}\upharpoonright n, \bar{d}\upharpoonright n)$ for some $n<\omega$
and some paly $(\bar{}c, \bar{d})$ in which $\exists$ uses $\sigma$.
The set $I$ is a back and forth system from $\A$ to $\B$.
First putting $n=0$ in the definition of $I$ , we see that $I$ contains the pair of $0$ tuples
$(\langle \rangle, \langle \rangle).$ 
This establishes (ii). Next (iii) and (iv) express that $\sigma$ 
tells player $\exists$ what to do at each step of this game. 
And finally $(i)$ holds because the strategy of $\sigma$ is winning. 
In the other direction, suppose that there exists a back 
and forth system $I$ from $\A$ to $\B$.
Define the set $I^*$ of maps as above, and choose an arbitrary well ordering of $I^*$. 
Consider the following
strategy $\sigma$ for player $\exists$ in the game $EF_{\omega}(\A,\B)$.
At each step if the play is so far $(\bar{a}, \bar{b})$ and 
$\forall$ has just chosen an element $c$ from $\A$, find the first map
$f$ in $I^*$ such that $\bar{a}$ and $c$ 
are in the domain of $f$ and $f(\bar{a})=f(\bar{b})$
and then choose $d$ to be $fc$, likewise in the other direction.
\end{demo}
This strategy makes $\exists$ win.
Coming back to our algebras we have:

\begin{theorem} 
\begin{enumroman}
\item $\exists$ has a winning strategy in $EEF_{\omega}(\A(n),\B(n))$.
\item $\A(n)\equiv_{\infty, \omega} \B(n).$
\end{enumroman}
\end{theorem}
\begin{proof} Both $\A(n)$ and $\B(n)$ are atomic . So $\A(n)$ and $\B(n)$
are identical in all components except for the components ''coloured " 
by $1_{u}$, $u\in T_n=V\sim J$ beneath which $\A(n)$ has 
uncountably many atoms and $\B(n)$ has countably many atoms.
Now for the game. 
At each step, if the play so far $(\bar{a}, \bar{b})$ and $\forall$ chooses an atom $a$ 
in one of the substructures, we have one of two case. 
Either $a.1_u=a$ for some $u\notin T_n$
in which case
$\exists$ chooses the same atom in the other structure. 
Else $a\leq 1_{u}$ for some $u\in T_n.$ 
Then
$\exists$ chooses a new atom below $1_{u}$ 
(distinct from $a$ and all atoms played so far.)
This is possible since there finitely many atoms in 
play and there are infinitely many atoms below
$1_{u}$.
This strategy makes $\exists$ win. 
Let $J$ be a back and forth system which exists by Theorem 6 and (i).   
Order $J$ by reverse inclusion, that is $f\leq g$ 
if $f$ extends $g$. $\leq$ is a partial order on $J$.
For $g\in J$, let $[g]=\{f\in J: f\leq g\}$. Then $\{[g]: g\in J\}$ is the base of a 
topology on 
$J.$ Let $\C$ be the complete  
Boolean algebra of regular open subsets of $J$ with respect to the topology 
defined on $J.$
Form the boolean extension $\M^{\C}.$
We want to define an isomorphism in $\M^{\C}$ of $\breve{\A}$ to 
$\breve{\B}.$
We shall use the following for $s\in \M^{\C}$, (1):
$$||(\exists x\in \breve{s})\phi(x)||=\sum_{a\in s}||\phi(\breve{a})||.$$
Define $G$ by (2):
$$||G(\breve{a},\breve{b})||=\{f\in {J}: f(a)=b\}$$
for $c\in \A$ and $d\in \B$.
If the right-hand side,  is not empty, that is it contains a function $f$, then let
$f_0$ be the restriction of $f$ to the substructure of $\A$ generated by $\{a\}$.
Then $f_0\in J.$ Also $$\{f\in J:  f(c)=d\}=[f_0]\in \C.$$
$G$ is therefore a $\C$-valued relation. Now let $u,v\in \M$.
Then 
$$||\breve{u}=\breve{v}||=1\text { iff }u=v,$$ 
and 
$$||\breve{u}=\breve{v}||=0\text { iff } u\neq v$$
Therefore
$$||G(\breve{a},\breve{b})\land G(\breve{a},\breve{c})||\subseteq ||\breve{b}=\breve{c}||.$$
for $a\in \A$ and $b,c\in \B.$
Therefore ``$G$ is a function." is valid. 
It is one to one because its converse is also a function.
(This can be proved the same way).
Finally we show that  that $\A(n)\equiv_{\infty\omega}\B(n)$ using "soft model theory" as follows:
Form a boolean extension $\M^*$ of $\M$
in which the cardinailities of $\A(n)$ and $\B(n)$ collapse to 
$\omega$.  Then $\A(n)$ and $\B(n)$ are still back and forth equivalent in $\M^*.$
Then $\A(n)\equiv_{\infty\omega}\B(n)$ in $\M^*$, and hence also in $\M$
by absoluteness of $\models$.
\end{proof}

\end{document}